\documentclass{article}

\usepackage{fullpage,amsthm,amsmath,amssymb,amsbsy,amsfonts}

\title{\bf Sharp estimates for the gradient of the generalized Poisson integral for a half-space}

\author{\sc{Gershon Kresin$^a\!\!$}
\thanks{e-mail: kresin@ariel.ac.il}$\;\;$
and \sc{Vladimir Maz'ya$^b$}
\thanks{e-mail: vladimir.mazya@liu.se}$\;\;$
\\ \\
{\it{$^a$Department of Mathematics, Ariel University, Ariel 40700, Israel}}\\
{\it{$^b$Department of Mathematical Sciences, University of Liverpool,
M$\&$O Building, Liverpool,}}\\ 
{\it{L69 3BX, UK; Department of Mathematics, Link\"oping University,SE-58183 Link\"oping,}}\\ 
{\it{{\hskip -145mm}Sweden}}
}
{ \date\ }

\numberwithin{equation}{section}

\newtheorem{theorem}{Theorem}
\newtheorem{proposition}[theorem]{Proposition}

\newenvironment{remark}{{\bf Remark}}

\newcommand{\bs}{\boldsymbol}

\begin{document}
\maketitle
\large
\centerline{\sl Dedicated to Vakhtang Kokilashvili on the occasion of his 80th birthday}
\vspace{10mm}

{\bf Abstract.} A representation of the sharp coefficient in a pointwise estimate for the gradient 
of the generalized Poisson integral of a function $f$ on ${\mathbb R}^n$ is obtained under the assumption that
$f$ belongs to  $L^p$. The explicit value of the coefficient is found for the cases $p=1$ and $p=2$.
\\
\\ 
\setcounter{equation}{0}
\section{Introduction}

In the paper \cite{KM2} (see also \cite{KM5}) a representation for the sharp coefficient
${\mathcal K}_p( x)$ in the inequality
$$
\left |\nabla u(x)\right |\leq {\mathcal K}_p( x)\big|\!\big |u \big |\!\big |_p
$$
was found, where $u$ is harmonic function in the half-space
${\mathbb R}^{n+1} _{+}=\big \lbrace x=(x', x_{n+1}): x'\in {\mathbb R}^n, x_{n+1} > 0\big \rbrace$,
represented by the Poisson integral with
boundary values in $L^p({\mathbb R}^n)$,  $||\cdot ||_p$ is the norm in
$L^p({\mathbb R}^n)$, $1\leq p \leq \infty$, $x \in {\mathbb R}^{n+1} _{+}$. It was shown that
$$
{\mathcal K}_p( x)=\frac{K_p}{x_{n+1}^{(n+p)/p}}
$$
and explicit formulas for $K_1$ and $K_2$ were given. Namely,
$$
K_1=\frac{2 n}{\omega_{n+1}}
\;,\;\;\;\;\;\;\;\;\;\;\;
K_2=\sqrt{\frac{n(n+1)}{2^{n+1} \omega_{n+1}}}\;,
$$
where $\omega_n=2\pi^{n/2}/\Gamma(n/2)$ is the area of the unit sphere in ${\mathbb R}^n$.

In \cite{KM2} it was shown that the sharp coefficients
in pointwise estimates for the absolute value of the normal derivative and 
the modulus of the gradient of a harmonic function in the half-space coincide for the case $p=1$ as well as for the case $p=2$.

Similar results for the gradient and the radial derivative of a harmonic function in the multidimensional ball with 
boundary values from $L^p$ for $p=1, 2$ in \cite{KM3} were obtained. 

Thus, the $L^1, L^2$-analogues of Khavinson's problem \cite{KHAV} were solved in \cite{{KM2},{KM3}} 
for harmonic functions in the multidimensional half-space and the ball.

We note that explicit sharp coefficients in the inequality for the first derivative of analytic function in the half-plane and the disk 
with boundary values of the real-part from $L^p$ in \cite{KM1,KM4,K1} were found.

\smallskip
In this paper we treat a generalization of the problem considered in our work \cite{KM2}.
Here we consider the generalized Poisson integral  
$$
u_{_f}(x)=k_{n, \alpha }\int _{{\mathbb R}^n}\frac{x_{n+1}^\alpha}{|y-x|^{n+\alpha}}\;f(y')dy'
$$
with $f\in L^p({\mathbb R}^n)$, $\alpha >-(n/p)$, $1\leq  p\leq \infty $, where $x\in{\mathbb R}^{n+1} _{+}$, 
$y=(y', 0)$, $y' \in {\mathbb R}^n$, and $k_{n, \alpha }$ is a normalization constant. 
In the case $\alpha=1$ the last integral coincides with the Poisson integral for a half-space.

\smallskip
In Section 2 we obtain a representation for the sharp coefficient ${\mathcal C}_p(x)$
in the inequality
$$
\left |\nabla u_{_f}(x)\right |\leq {\mathcal C}_p( x)\big|\!\big |f \big |\!\big |_p\;,
$$
where 
$$
{\mathcal C}_p( x)=\frac{C_p}{x_{n+1}^{(n+p)/p}}
$$
and the constant $C_p$ is characterized in terms of an extremal problem on the unit sphere ${\mathbb S}^n$ in ${\mathbb R}^{n+1}$.

\smallskip
In Section 3 we reduce this extremal problem to that of finding of the supremum
of a certain double integral, depending on a scalar parameter and show that
$$
C_1=k_{n, \alpha }n
$$
if $-n<\alpha \leq n$, and
$$
C_2=\sqrt{\omega_{n-1}}k_{n, \alpha }\left \{\frac{\sqrt{\pi}(n+\alpha)n(n+2)\Gamma\left (\frac{n}{2}-1 \right )
\Gamma\left (\frac{n}{2}+\alpha \right )}{8(n+1+\alpha)\Gamma(n+\alpha)}\right \}^{1/2}
$$
if $-(n/2)<\alpha \leq n(n+1)/2$.

\smallskip
It is shown that the sharp coefficients in pointwise estimates for the absolute value of the normal derivative and 
the modulus of the gradient of the generalized Poisson integral for a half-space coincide in the case $p=1$ as well as in the case $p=2$.

\section{Representation for the sharp constant in inequality for the gradient in terms of an extremal problem on the unit sphere}

We introduce some notation used henceforth.
Let ${\mathbb R}^{n+1} _{+}=\big \lbrace x=(x', x_{n+1}): x'=(x_1,\dots,x_n)
\in {\mathbb R}^n, x_{n+1} > 0\big \rbrace$, ${\mathbb S}^n=
\{ x \in {\mathbb R}^{n+1}: |x|=1 \}$,
${\mathbb S}^n_+= \{ x \in {\mathbb R}^{n+1}: |x|=1,\; x_{n+1} >0 \}$ and
${\mathbb S}^n_{-}= \{ x \in {\mathbb R}^{n+1}: |x|=1,\; x_{n+1}<0 \}$.
Let $\bs e_\sigma $ stand for the
$n+1$-dimensional unit vector joining the origin to a point $\sigma $
on the sphere ${\mathbb S}^n$.

By $||\cdot ||_p$ we denote the norm in the space $L^p({\mathbb R}^n)$, that is
$$
|| f||_p=\left \{\int _ {{\mathbb R}^n } |f(x')|^p \;dx' \right \}^{1/p},
$$
if $1\leq  p< \infty $, and $|| f||_\infty =\mbox{ess}\;\sup \{ | f(x') |:
x' \in {\mathbb R}^n \}$.

Let the function in ${\mathbb R}^{n+1}_{+}$ be represented as the generalized Poisson integral
\begin{equation} \label{EH_1}
u_{_f}(x)=k_{n, \alpha }\int _{{\mathbb R}^n}\frac{x_{n+1}^\alpha}{|y-x|^{n+\alpha}}\;f(y')dy'
\end{equation}
with $f\in L^p({\mathbb R}^n)$, $1\leq  p\leq \infty $, where $y=(y', 0)$, $y' \in {\mathbb R}^n$, 
\begin{equation} \label{EQU_01}
k_{n, \alpha }=\left \{ \int _{{\mathbb R}^n}\frac{x_{n+1}^\alpha}{|y-x|^{n+\alpha}}dy' \right \}^{-1}=
\frac{\Gamma\left (\frac{n+\alpha}{2} \right )}{\pi^{n/2}\Gamma\left (\frac{\alpha}{2}\right )}\;, 
\end{equation}
and
\begin{equation} \label{EQU_1}
\alpha >-\frac{n}{p}\;. 
\end{equation}

Now, we find a representation for the best coefficient ${\mathcal C}_p( x; \bs z)$
in the inequality for the absolute value of derivative of $u_{_f}(x)$ 
in an arbitrary direction $\bs z \in {\mathbb S}^n$, $x\in {\mathbb R}^{n+1}_{+}$.
In particular, we obtain a formula for the constant in a similar inequality for the
modulus of the gradient.

\begin{proposition} \label{P_1}
Let $x $ be  an arbitrary point in ${\mathbb R}^{n+1} _+$ and let $\bs z\in {\mathbb S}^n$. 
The sharp coefficient ${\mathcal C}_p (x; \bs z)$ in the inequality
$$
|\left ( \nabla u_{_f}(x), \bs z \right ) |\leq {\mathcal C}_p( x; \bs z)
\big|\!\big |f \big |\!\big |_p
$$
is given by
\begin{equation} \label{EH_2A}
{\mathcal C}_p( x; \bs z)= \frac{C_p(\bs z)}{x_{n+1}^{(n+p)/p}},
\end{equation}
where
\begin{equation} \label{EH_2B}
C_1(\bs z)=k_{n, \alpha }\sup _{\sigma \in {\mathbb S}^n_+}
\big |\big (\alpha\bs e_{n+1} -(n+\alpha)(\bs e_{\sigma}, \bs e_{n+1})\bs e_{\sigma},\; \bs z \big )\big |
\big (\bs e_{\sigma}, \bs e_{n+1} \big )^{n+\alpha} ,
\end{equation}
\begin{equation} \label{EH_3DA}
C_p(\bs z)\!=\!k_{n, \alpha}\!
\left \{ \int _ {{\mathbb S}^n_+ }\!\!
\big |\!\big (\alpha\bs e_{n+1} \!-\!(n\!+\!\alpha)(\bs e_{\sigma}, \bs e_{n+1})\bs e_{\sigma}, \bs z \big )\!\big |^{\frac{p}{p\!-\!1}}
\big (\bs e_{\sigma}, \bs e_{n+1} \big )^{\frac{(\alpha\!-\!1)p\!+\!n\!+\!1}{p-1}}d\sigma \!\right \}^{\frac{p\!-\!1}{p}}
\end{equation}
for $1<p<\infty $,
and
\begin{equation} \label{EH_3H}
C_\infty(\bs z)=k_{n, \alpha}
\int _ {{\mathbb S}^n_+ }
\big |\big (\alpha\bs e_{n+1} -(n+\alpha)(\bs e_{\sigma}, \bs e_{n+1})\bs e_{\sigma},\; \bs z \big )\big |
\big (\bs e_{\sigma}, \bs e_{n+1} \big )^{\alpha-1}\;d\sigma .
\end{equation}

In particular, the sharp coefficient ${\mathcal C}_p (x)$ in the inequality
$$
\left |\nabla u_{_f}(x)\right |\leq {\mathcal C}_p( x)\big|\!\big |f \big |\!\big |_p
$$
is given by
\begin{equation} \label{EH_3J}
{\mathcal C}_p(x)=\frac{C_p}{x_{n+1}^{(n+p)/p}},
\end{equation}
where
\begin{equation} \label{EH_4J}
C_p=\sup _{|\bs z|=1}C_p(\bs z).
\end{equation}
\end{proposition}
\begin{proof} Let $x=(x', x_{n+1})$ be a fixed point in ${\mathbb R}^{n+1}_+$. 
The representation (\ref{EH_1}) implies
$$
\frac{\partial u_{_f}}{\partial x_i}=k_{n, \alpha }\int _{{\mathbb R}^n}
\left [ \frac{\delta_{ni}\alpha x_{n+1}^{\alpha-1}}{|y-x|^{n+\alpha}} + \frac{(n+\alpha) 
x_{n+1} ^\alpha(y_i-x_i)}{|y-x|^{n+2+\alpha}}  \right ]f(y')dy',
$$
that is
\begin{eqnarray*}
\nabla u_{_f}(x)&=&k_{n, \alpha }x_{n+1}^{\alpha-1}\int _{{\mathbb R}^n}
\left [\; \frac{\alpha\bs e_{n+1}}{|y-x|^{n+\alpha}} + \frac{(n+\alpha) x_{n+1} (y-x)}{|y-x|^{n+2+\alpha}}\; \right ]f(y')dy'\\
& &\\
&=&k_{n, \alpha }x_{n+1}^{\alpha-1}\int _{{\mathbb R}^n}
\frac{\alpha\bs e_{n+1}  - (n+\alpha)(\bs e_{xy}, \bs e_{n+1}) \bs e_{xy}}{|y-x|^{n+\alpha}}\; f(y')dy',
\end{eqnarray*}
where $\bs e_{xy}=(y-x)|y-x|^{-1}$.
For any $\bs z \in {\mathbb S}^n$,
\begin{equation} \label{EH_DIR}
(\nabla u_{_f}(x), \bs z)=
k_{n, \alpha }x_{n+1}^{\alpha-1}\int _{{\mathbb R}^n}
\frac{(\alpha\bs e_{n+1}  - (n+\alpha)\bs e_{xy}, \bs e_{n+1})\bs e_{xy},\; \bs z)}{|y-x|^{n+\alpha}}\; f(y')dy'.
\end{equation}
Hence,
\begin{equation} \label{EH_7AD}
{\mathcal C}_1( x; \bs z)=k_{n, \alpha }x_{n+1}^{\alpha-1}
\sup _{y \in {\partial\mathbb R}^n}
\frac{|(\alpha\bs e_{n+1}  - (n+\alpha)(\bs e_{xy}, \bs e_{n+1})\bs e_{xy},\; \bs z)|}{|y-x|^{n+\alpha}}\;,
\end{equation}
and
\begin{equation} \label{EH_6A}
{\mathcal C}_p( x; \bs z)=k_{n, \alpha }x_{n+1}^{\alpha-1}
\left \{\int _{{\mathbb R}^n}
\frac{\big |\big (\alpha\bs e_{n+1}  - (n+\alpha)(\bs e_{xy}, \bs e_{n+1})\bs e_{xy}, \bs z \big )\big |^q }
{|y-x|^{(n+\alpha)q}}\;dy' \right \}^{1/q}
\end{equation}
for $1<p \leq \infty $, where $p^{-1}+q^{-1}=1$.

Taking into account the equality
\begin{equation} \label{EH_COS}
\frac{x_{n+1}}{|y-x|}=(\bs e_{xy}, -\bs e_{n+1}),
\end{equation}
by (\ref{EH_7AD}) we obtain
\begin{eqnarray*}
{\mathcal C}_1( x; \bs z)
&=&k_{n, \alpha }x_{n+1}^{\alpha-1}\sup _{y \in \partial{\mathbb R}^n}
\frac{|(\alpha\bs e_{n+1}  - (n+\alpha)(\bs e_{xy}, \bs e_{n+1})\bs e_{xy},\; \bs z)|}{x_{n+1}^{n+\alpha}}
\left (\frac{x_{n+1}}{|y-x|}\right )^{n+\alpha} \\
& &\\
&=&\frac{k_{n, \alpha }}{x_{n+1}^{{n+1}}}\sup _{\sigma \in {\mathbb S}^n_-}
\big |\big (\alpha\bs e_{n+1} -(n+\alpha)(\bs e_{\sigma}, \bs e_{n+1})\bs e_{\sigma},\; \bs z \big )\big |\big (\bs e_{\sigma}, -\bs e_{n+1} 
\big )^{n+\alpha}.
\end{eqnarray*}
Replacing here $\bs e_\sigma$ by $-\bs e_\sigma$, we arrive at (\ref{EH_2A})
for $p=1$  with the sharp constant (\ref{EH_2B}).

Let $1<p \leq \infty $. Using (\ref{EH_COS}) and the equality
$$
\frac{1}{|y-x|^{(n+\alpha) q}}=\frac{1}{x_{n+1}^{(n+\alpha)q-n}}\left (\frac{x_{n+1}}{|y-x|} \right )^{(n+\alpha)q-n-1}
\frac{x_{n+1}}{|y-x|^{n+1}}\;,
$$
and replacing  $q$ by $p/(p-1)$ in (\ref{EH_6A}), we conclude that (\ref{EH_2A}) holds
with the sharp constant
$$
C_p(\bs z)=k_{n, \alpha }
\left \{ \int _ {{\mathbb S}^n_- }
\big |\big (\alpha\bs e_{n+1} -(n+\alpha)(\bs e_{\sigma}, \bs e_{n+1})\bs e_{\sigma},\; \bs z \big )\big |^{\frac{p}{p-1}}
\big (\bs e_{\sigma}, -\bs e_{n+1} \big )^{\frac{(\alpha-1) p+n+1}{p-1}}\;d\sigma \right \}^{\frac{p-1}{p}},
$$
where ${\mathbb S}^n_- =\{ \sigma \in {\mathbb S}^n: (\bs e_\sigma , \bs e_{n+1})<0 \}$.
Replacing here $\bs e_\sigma$ by $-\bs e_\sigma$, 
we arrive at (\ref{EH_3DA}) for $1<p<\infty$ and at (\ref{EH_3H}) for $p=\infty$.

By (\ref{EH_DIR}) we have
$$
\big |\nabla u_{_f}(x)\big |=k_{n, \alpha }x_{n+1}^{\alpha-1}\sup _{|\bs z|=1}\int _{{\mathbb R}^n}
\frac{(\alpha\bs e_{n+1}  - (n+\alpha)(\bs e_{xy}, \bs e_{n+1})\bs e_{xy},\; \bs z)}{|y-x|^{n+\alpha}}\; f(y')dy'.
$$
Hence, by the permutation of suprema, (\ref{EH_6A}), (\ref{EH_7AD}) and (\ref{EH_2A}),
\begin{eqnarray} \label{EH_6AJ}
{\mathcal C}_p( x)&=&k_{n, \alpha }x_{n+1}^{\alpha-1}\sup _{|\bs z|=1}
\left \{\int _{{\mathbb R}^{n+1}}
\frac{\big |\big (\alpha\bs e_{n+1}  - (n+\alpha)(\bs e_{xy}, \bs e_{n+1})\bs e_{xy}, \bs z \big )\big |^ q }
{|y-x|^{(n+\alpha)q}}\;dy' \right \}^{1/q} \nonumber\\
& &\nonumber\\
&=&\sup _{|\bs z|=1}{\mathcal C}_p( x; \bs z)=\sup _{|\bs z|=1}C_p(\bs z) x_{n+1}^{-(n+p)/p}
\end{eqnarray}
for $1<p \leq \infty $, and
\begin{eqnarray} \label{EH_7ADJ}
{\mathcal C}_1( x)&=&k_{n, \alpha }x_{n+1}^{\alpha-1}\sup _{|\bs z|=1}\;\sup _{y \in \partial{\mathbb R}^n}
\frac{|(\alpha\bs e_{n+1}  - (n+\alpha)(\bs e_{xy}, \bs e_{n+1})\bs e_{xy},\; \bs z)|}{|y-x|^{n+\alpha}}\nonumber\\
& &\nonumber\\
&=&\sup _{|\bs z|=1}{\mathcal C}_1( x; \bs z)=\sup _{|\bs z|=1}C_1(\bs z) x_{n+1}^{-(n+1)}.
\end{eqnarray}
Using the notation (\ref{EH_4J}) in (\ref{EH_6AJ}) and (\ref{EH_7ADJ}), we arrive at (\ref{EH_3J}).
\end{proof}

\begin{remark}. 
Formula (\ref{EH_3DA}) for the coefficient $C_p(\bs z), 1<p< \infty,$
can be written with the integral over the whole sphere ${\mathbb S}^n$
in ${\mathbb R}^{n+1}$,
$$
C_p(\bs z)=\frac{k_{n, \alpha }}{2^{(p-1)/p}}\!
\left \{ \!\int _ {{\mathbb S}^n }\!\!
\big |\!\big (\alpha\bs e_{n+1} \!-\!(n\!+\!\alpha)(\bs e_{\sigma}, \bs e_{n+1})\bs e_{\sigma}, \bs z \big )\!\big |^{\frac{p}{p\!-\!1}}
\big (\bs e_{\sigma}, \bs e_{n+1} \big )^{\frac{(\alpha\!-\!1)p\!+\!n\!+\!1}{p-1}}d\sigma \!\right \}^{\frac{p\!-\!1}{p}}
$$
A similar remark relates (\ref{EH_3H}):
\begin{equation} \label{EH_3HW}
C_\infty(\bs z)=\frac{k_{n, \alpha }}{2}
\int _ {{\mathbb S}^n}
\big |\big (\alpha\bs e_{n+1} -(n\!+\!\alpha)(\bs e_{\sigma}, \bs e_{n+1})\bs e_{\sigma},\; \bs z \big )\big |
\big |\big (\bs e_{\sigma}, \bs e_{n+1} \big )\big |^{\alpha-1}\;d\sigma , 
\end{equation}
as well as formula (\ref{EH_2B}):
$$
C_1(\bs z)=k_{n, \alpha }\sup _{\sigma \in {\mathbb S}^n}
\big |\big (\alpha\bs e_{n+1} -(n\!+\!\alpha)(\bs e_{\sigma}, \bs e_{n+1})\bs e_{\sigma},\; 
\bs z \big )\big |\big |\big (\bs e_{\sigma}, \bs e_{n+1} \big )\big |^{n+\alpha}\;.
$$
\end{remark}

\section{Reduction of the extremal problem to finding of the supremum by parameter of a double integral. 
The cases $p=1$ and $p=2$}

The next assertion is based on the representation for $C_p$,
obtained in Proposition \ref{P_1}.

\begin{proposition} \label{P_2} Let $f \in L^p({\mathbb R}^n)$,
and let  $x $ be  an arbitrary point in ${\mathbb R}^{n+1} _+$. 
The sharp coefficient ${\mathcal C}_p (x)$ in the inequality
\begin{equation} \label{EH_2_3}
|\nabla u_{_f}(x)|\leq {\mathcal C}_p( x)\big|\!\big |f \big |\!\big |_p
\end{equation}
is given by
\begin{equation} \label{EH_2A_3A}
{\mathcal C}_p( x)=\frac{C_p}{x_{n+1}^{(n+p)/p}}\;,
\end{equation}
where
\begin{equation} \label{EH_1ABC}
C_p\!=\!(\omega_{n-1})^{(p-1)/p}k_{n, \alpha }
\sup _{\gamma \geq 0}\;\frac{1}{\sqrt{1+\gamma^2}}\left \{ \int _ {0}^{\pi}d\varphi
\int _ {0}^{\pi/2}{\mathcal F}_{n,p}(\varphi, \vartheta ; \gamma) \;
 d\vartheta\right \}^\frac{p-1}{p},
\end{equation}
if  $1<p<\infty $. Here
\begin{equation} \label{EH_1AC}
{\mathcal F}_{n,p}(\varphi, \vartheta ; \gamma)
=\big |{\mathcal G}_{n}(\varphi, \vartheta ; \gamma) \big |^{p / (p-1)}
\cos^{((\alpha-1)p+n+1) / (p-1)}\vartheta \sin^{n-1}\vartheta \sin^{n-2}\varphi
\end{equation}
with
\begin{equation} \label{E_AUX}
{\mathcal G}_{n}(\varphi, \vartheta ; \gamma)= \big ((n+\alpha)\cos^2 \vartheta -\alpha\big )+\gamma(n+\alpha)\cos \vartheta
\sin \vartheta\cos \varphi \;.
\end{equation}

In addition,
\begin{equation} \label{EH_2BAA}
C_1=k_{n, \alpha }n
\end{equation}
if $-n<\alpha \leq n$.

In particular,
$$
C_2=\sqrt{\omega_{n-1}}k_{n, \alpha }\left \{\frac{\sqrt{\pi}(n+\alpha)n(n+2)\Gamma\left (\frac{n}{2}-1 \right )
\Gamma\left (\frac{n}{2}+\alpha \right )}{8(n+1+\alpha)\Gamma(n+\alpha)}\right \}^{1/2}
$$
for $-(n/2)<\alpha \leq n(n+1)/2$.

For $p=1$ and $p=2$ the coefficient ${\mathcal C}_p( x)$ is sharp in conditions of the Proposition also in
the weaker inequality obtained from $(\ref{EH_2_3})$ by replacing $\nabla u_{_f}$ by
$\partial u_{_f} /\partial x_{n+1}$.
\end{proposition}

\begin{proof} The equality (\ref{EH_2A_3A}) was proved in Proposition \ref{P_1}.

(i) Let $p=1$. Using (\ref{EH_2B}), (\ref{EH_4J}) and the permutability of two
suprema, we find
\begin{eqnarray} \label{EH_2BP}
C_1&=&k_{n, \alpha}\sup_{|\bs z|=1}\sup _{\sigma \in {\mathbb S}^n_+ }
\big |\big (\alpha\bs e_{n+1} -(n+\alpha)(\bs e_{\sigma}, \bs e_{n+1})\bs e_{\sigma},\; \bs z \big )\big |\big (\bs e_{\sigma}, \bs e_{n+1} \big )^{n+\alpha} \nonumber\\
& &\nonumber\\
&=&k_{n, \alpha}\sup _{\sigma \in {\mathbb S}^n_+ }
\big |\alpha\bs e_{n+1} -(n+\alpha)(\bs e_{\sigma}, \bs e_{n+1})\bs e_{\sigma}\big |\big (\bs e_{\sigma}, \bs e_{n+1} \big )^{n+\alpha}  \;.
\end{eqnarray}
Taking into account the equality
\begin{eqnarray*}
& &\big |\alpha\bs e_{n+1} -(n+\alpha)(\bs e_{\sigma}, \bs e_{n+1})\bs e_{\sigma}\big |\\
& &\\
& &=\Big (\alpha\bs e_{n+1} -(n+\alpha)(\bs e_{\sigma}, \bs e_{n+1})\bs e_{\sigma}, 
\;\alpha\bs e_{n+1} -(n+\alpha)(\bs e_{\sigma}, \bs e_{n+1})\bs e_{\sigma} \Big )^{1/2}\\
& &\\
& &=\Big (\alpha^2+((n+\alpha)^2-2\alpha(n+\alpha) )(\bs e_{\sigma},\; \bs e_{n+1})^2 \Big )^{1/2},
\end{eqnarray*}
and using (\ref{EQU_1}), (\ref{EH_2BP}), we arrive at the sharp constant (\ref{EH_2BAA}) for $-n<\alpha\leq n$.

Furthermore, by (\ref{EH_2B}),
$$
C_1(\bs e_{n+1})=k_{n, \alpha}\sup _{\sigma \in {\mathbb S}^n_+ }
\big |\alpha -(n+\alpha)(\bs e_{\sigma}, \bs e_{n+1})^2 |\big (\bs e_{\sigma}, \bs e_{n+1} \big )^{n+\alpha} \geq
k_{n, \alpha}n.
$$
Hence, by $C_1 \geq C_1(\bs e_{n+1})$ and by (\ref{EH_2BAA}) we obtain $C_1=C_1(\bs e_{n+1})$,
which completes the proof in the case $p=1$.

\medskip
(ii) Let $1<p<\infty $. Since the integrand in (\ref{EH_3DA}) does not
change when $\bs z \in {\mathbb S}^n$ is replaced by $-\bs z$,
we may assume that $z_{n+1}=(\bs e_{n+1}, \bs z) > 0$ in (\ref{EH_4J}).

Let $\bs z'=\bs z-z_{n+1}\bs e_{n+1}$. Then $(\bs z', \bs e_{n+1})=0$ and hence
$z^2_{n+1}+|\bs z'|^2=1$.
Analogously, with
$\sigma=(\sigma_1,\dots,\sigma_n,\sigma_{n+1}) \in {\mathbb S}^n_+$,
we associate the vector $\bs \sigma '=\bs e_\sigma -\sigma_{n+1}\bs e_{n+1}$.

Using the equalities $(\bs \sigma ', \bs e_{n+1})=0$,
$\sigma_{n+1} =\sqrt{1-|\bs \sigma '|^2}$
and $(\bs z', \bs e_{n+1})=0$, we find an expression for
$(\alpha\bs e_{n+1} -(n+\alpha)(\bs e_{\sigma}, \bs e_{n+1})\bs e_{\sigma},\; \bs z \big )$
as a function of $\bs \sigma'$:
\begin{eqnarray}
& &(\alpha\bs e_{n+1} -(n+\alpha) (\bs e_{\sigma}, \bs e_{n+1})\bs e_{\sigma},\; \bs z \big )=
\alpha z_{n+1}-(n+\alpha)\sigma_{n+1}\big ( \bs e_{\sigma}, \bs z \big )\nonumber\\
& &=\alpha z_{n+1}-(n+\alpha)\sigma_{n+1}
\big ( \bs \sigma '+\sigma_{n+1} \bs e_{n+1},\; \bs z'+z_{n+1}\bs e_{n+1} \big )\nonumber\\
& &=\alpha z_{n+1}-(n+\alpha)\sigma_{n+1}
\big [ \big ( \bs \sigma ', \bs z' \big )+z_{n+1} \sigma_{n+1}\big ]\nonumber\\
& &=-\big [(n+\alpha)(1-|\bs\sigma '|^2) -\alpha \big ]z_{n+1}-(n+\alpha)\sqrt{1-|\bs\sigma '|^2}\;
\big ( \bs \sigma ', \bs z' \big ) .
\label{EH_B}
\end{eqnarray}

Let ${\mathbb B}^n=\{ x'=(x_1,\dots,x_n)\in {\mathbb R}^n: |x'|< 1 \}$.
By (\ref{EH_3DA}) and (\ref{EH_B}), taking into account that
$d\sigma=d\sigma '/\sqrt{1-|\bs\sigma '|^2}$, we may write (\ref{EH_4J}) as
\begin{eqnarray} \label{EH_B1C}
\hspace{-15mm}& &C_p=k_{n, \alpha}\sup _{\bs z \in {\mathbb S}^n_+}
\left \{ \int _ {{\mathbb B}^n}
\frac{{\mathcal H}_{n, p}\big  ( |\bs \sigma '|,  (\bs \sigma ',  \bs z') \big  )
\big ( 1-|\bs\sigma '|^2\big )^{(\alpha p+n+1)/2(p-1)}}
{\sqrt{1-|\bs\sigma '|^2}}\;d\sigma ' \right \}^{\frac{p-1}{p}} \nonumber \\
\hspace{-15mm}& &\nonumber \\
\hspace{-15mm}& &=k_{n, \alpha}\sup _{\bs z \in {\mathbb S}^n_+}\!\!
\left \{\! \int _ {{\mathbb B}^n}
\!\!{\mathcal H}_{n, p}\big  ( |\bs \sigma '|,  (\bs \sigma ',  \bs z') \big  )
\big ( 1\!-\!|\bs\sigma '|^2\big )^{((\alpha-2)p+n+2)/2(p-1)}d\sigma '\! \right \}^{\frac{p-1}{p}},
\end{eqnarray}
where
\begin{equation} \label{E_FH}
{\mathcal H}_{n, p}\big  ( |\bs \sigma '|,  (\bs \sigma ',  \bs z') \big  )
=\Big | \big [(n+\alpha)(1-|\bs\sigma '|^2) -\alpha \big ]z_{n+1}\!+\!(n+\alpha )
\sqrt{1-|\bs\sigma '|^2}\;\big ( \bs \sigma ', \bs z' \big )\Big |^{p/(p-1)} .
\end{equation}

Using the well known formula (see e.g. \cite{PBM}, \textbf{3.3.2(3)}),
$$
\int_{B^{n}}g\big (|\bs x|, (\bs a, \bs x)\big )dx=\omega_{n-1}\int_0^1 r^{n-1} dr
\int_0^\pi g\big ( r, |\bs a|r \cos \varphi \big )\sin ^{n-2}\varphi \;d\varphi \;,
$$
we obtain
\begin{eqnarray*} 
& &\hspace{-17mm}\int _ {{\mathbb B}^n}
{\mathcal H}_{n, p}\big  ( |\bs \sigma '|,  (\bs \sigma ',  \bs z') \big  )
\big ( 1-|\bs\sigma '|^2\big )^{((\alpha-2)p+n+2)/2(p-1)}\;d\sigma ' \nonumber\\
& &\hspace{-17mm}\\
& &\hspace{-17mm}=\omega_{n-1}\!\int^{1}_{0}\! r^{n-1}\big ( 1\!-\!r^2\big )^{((\alpha-2)p+n+2)/2(p-1)} dr
\!\int^{\pi}_{0}\!\!{\mathcal H}_{n, p}
\big  (r, r|\bs z'|\cos  \varphi \big  )\sin ^{n-2}\varphi d\varphi\;.
\end{eqnarray*}
Making the change of variable $r=\sin \vartheta $ in 
the right-hand side of the last equality, we find
\begin{eqnarray} \label{E_PBMD}
& &\hspace{-7mm}\int _ {{\mathbb B}^n}
{\mathcal H}_{n, p}\big  ( |\bs \sigma '|,  (\bs \sigma ',  \bs z') \big  )
\big ( 1-|\bs\sigma '|^2\big )^{\frac{(\alpha-2)p+n+2}{2(p-1)}}\;d\sigma ' \\
& &\hspace{-7mm}\nonumber\\
& &\hspace{-7mm}=\omega _{n-1}\int^{\pi}_{0}\!\! \sin ^{n-2}\varphi d\varphi \!\int^{\pi/2}_{0}\!\!
{\mathcal H}_{n, p}\big  (\sin \vartheta,\; |\bs z'|\sin \vartheta\cos  \varphi \big  )
\sin ^{n-1}\vartheta \cos^{\frac{(\alpha-1)p+n+1}{p-1}}\vartheta d\vartheta
\;,\nonumber
\end{eqnarray}
where, by (\ref{E_FH}),
$$
{\mathcal H}_{n, p}\big  (\sin \vartheta,\; |\bs z'|\sin \vartheta\cos  \varphi \big  )=
\Big |\big ( (n+\alpha )\cos^2 \vartheta -\alpha \big )z_{n+1}\!+\!(n+\alpha)|\bs z'|\cos \vartheta
\sin \vartheta\cos \varphi \Big |^{p/(p-1)}.
$$
Introducing here the parameter $\gamma =|\bs z'|/z_{n+1}$ and using
the equality $|\bs z'|^2+z^2_{n+1}=1$, we obtain
\begin{equation} \label{E_FHGH}
{\mathcal H}_{n, p}\big  (\sin \vartheta,\; |\bs z'|\sin \vartheta\cos  \varphi \big  )=
(1+\gamma^2)^{-p/2(p-1)}\big |{\mathcal G}_{n}(\varphi, \vartheta ; \gamma)\big |^{p/(p-1)},
\end{equation}
where ${\mathcal G}_{n}(\varphi, \vartheta ; \gamma)$ is given by (\ref{E_AUX}).

By (\ref{EH_B1C}), taking into account (\ref{E_PBMD}) and (\ref{E_FHGH}),
we arrive at (\ref{EH_1ABC}).

\medskip
(iii) Let $p=2$. By (\ref{EH_1ABC}), (\ref{EH_1AC}) and (\ref{E_AUX}),
\begin{equation} \label{EH_9ABC}
C_2= \sqrt{\omega_{n-1}}\;k_{n,\alpha}\;\sup _{\gamma \geq 0}\;\frac{1}
{\sqrt{1+\gamma^2}}\left \{ \int _ {0}^{\pi}d\varphi\int _ {0}^{\pi/2}\!\!
{\mathcal F}_{n,2}(\varphi, \vartheta ; \gamma)\; d\vartheta\right \}^{1/2},
\end{equation}
where
$$
\!{\mathcal F}_{n,2}(\varphi, \vartheta ; \gamma)\!=\!
\big [\big ( (n+\alpha )\cos^2 \vartheta \!-\!\alpha \big )\!+\!\gamma(n+\alpha )\cos \vartheta
\sin \vartheta\cos \varphi \big ]^{2}\! \cos^{n-1+2\alpha}\!\vartheta \!\sin^{n-1}\!\vartheta \sin ^{n-2}\!\varphi.
$$
The last equality and (\ref{EH_9ABC}) imply
\begin{equation} \label{EH_11ABCD}
C_2=\sqrt{\omega_{n-1}}\;k_{n,\alpha}\;\sup _{\gamma \geq 0}\;\frac{1}{\sqrt{1+\gamma^2}}\left \{{\mathcal I}_1+ \gamma^2 {\mathcal I}_2 \right \}^{1/2},
\end{equation}
where
\begin{eqnarray} \label{EH_9B}
\hspace{-15mm}& &{\mathcal I}_1=\int _ {0}^{\pi} \sin ^{n-2}\varphi\;d\varphi
\int _ {0}^{\pi/2}\big ((n+\alpha)\cos^2 \vartheta -\alpha\big )^2\sin^{n-1}
\vartheta \cos^{n-1+2\alpha} \vartheta \;d\vartheta \nonumber\\
\hspace{-15mm}& &\nonumber\\
\hspace{-15mm}& &=\frac{ \sqrt{\pi}n(n+2)(n+\alpha)\;\Gamma\left (\frac{n-1}{2} \right )\Gamma\left (\frac{n+2+\alpha}{2}
\right )}{4(n+2\alpha)(n+1+\alpha)\Gamma(n+\alpha)} 
\end{eqnarray}
and
\begin{eqnarray} \label{EH_9BA}
\hspace{-8mm}{\mathcal I}_2&=&(n+\alpha)^2\int _ {0}^{\pi}\sin ^{n-2}\varphi \cos^2 \varphi\; d\varphi
\int _ {0}^{\pi/2}\sin^{n+1}\vartheta \cos^{n+1+2\alpha} \vartheta
\;d\vartheta \nonumber\\
\hspace{-8mm}& &\nonumber\\
\hspace{-8mm}&=&\frac{\sqrt{\pi}\; (n+\alpha)\;\Gamma\left (\frac{n-1}{2} \right )
\Gamma\left (\frac{n+2+2\alpha}{2} \right )}{4(n+1+\alpha) \Gamma(n+\alpha)}.
\end{eqnarray}
By (\ref{EH_11ABCD}) we have
\begin{equation} \label{EQ_11ABCD}
C_2=\sqrt{\omega_{n-1}}\;k_{n,\alpha}\;\max \big \{ {\mathcal I}_1^{1/2},
{\mathcal I}_2^{1/2} \big  \}.
\end{equation}
Further, by (\ref{EH_9B}) and (\ref{EH_9BA}),
$$
\frac{{\mathcal I}_1}{{\mathcal I}_2}=\frac{n(n+2)}{n+2\alpha}.
$$
Therefore,
$$
\frac{{\mathcal I}_1}{{\mathcal I}_2}-1=\frac{n^2+n-2\alpha}{n+2\alpha}.
$$
Taking into account (\ref{EQ_11ABCD}) and that $n+2\alpha>0$ for $p=2$ by (\ref{EQU_1}), 
we see that inequality 
$$
\frac{{\mathcal I}_1}{{\mathcal I}_2}\geq 1
$$
holds for $\alpha\leq n(n+1)/2$. So, we arrive at the representation for $C_2$ with $-(n/2)<\alpha\leq n(n+1)/2$ 
given in formulation of the Proposition.

Since $\bs z \in {\mathbb S}^n$ and the supremum in $\gamma =|\bs z'|/z_{n+1}$ in
(\ref{EH_9ABC}) is attained for $\gamma=0$, we have $C_2=C_2(\bs e_{n+1})$ under requirements of the Proposition.
\end{proof}


\begin{thebibliography}{99}

\bibitem{KHAV} D. Khavinson, \textit{An extremal problem for harmonic functions in the ball}, 
Canad.  Math. Bull., \textbf{35} (1992), 218--220.

\bibitem{KM1} G. Kresin and V. Maz'ya, \textit{Sharp Real-Part Theorems. A Unified Approach},
Lect. Notes in Math., \textbf{1903}, Springer, Berlin, 2007.

\bibitem{KM2} G. Kresin and V. Maz'ya, \textit{Optimal estimates for the gradient of
harmonic functions in the multidimensional half-space}, Discrete and Continuous
Dynamical Systems, series B, \textbf{28} (2) (2010), 425-440.

\bibitem{KM3} G. Kresin and V. Maz'ya,  \textit{Sharp pointwise estimates for
directional derivatives of harmonic functions in a  multidimensional ball},
J. Math. Sc. (New York), \textbf{169}:2 (2010), 167--187.

\bibitem{KM4} G. Kresin and V. Maz'ya, \textit{Sharp  real-part  theorems  in  the  upper 
half-plane and similar estimates for harmonic functions}, J. Math. Sc. (New York),
\textbf{179}:1 (2011), 144--163.

\bibitem{KM5} G. Kresin and V. Maz'ya, \textit{Maximum Principles and Sharp 
Constants for Solutions of Elliptic and Parabolic Systems},
Math. Surveys and Monographs, \textbf{183}, Amer. Math. Soc., Providence, 
Rhode Island, 2012.

\bibitem {K1} G. Kresin, \textit{Sharp and maximized real-part estimates for derivatives of 
analytic functions in the disk}, Rendiconti Lincei - Matematica e Applicazioni, {\textbf 24}:1 (2013), 95--110.

\bibitem{PBM} A.P. Prudnikov, Yu.A. Brychkov and O.I. Marichev,
\textit{Integrals and Series, Vol. 1, Elementary Functions},
Gordon and Breach Sci. Publ., New York, 1986.

\end{thebibliography}
\end{document}